\documentclass[oneside,english]{amsart}
\usepackage[]{fontenc}
\usepackage[latin9]{inputenc}
\usepackage{color}
\usepackage{babel}
\usepackage{float}
\usepackage{amsthm}
\usepackage{amstext}
\usepackage{amssymb}
\usepackage[numbers]{natbib}
\usepackage[unicode=true,pdfusetitle,
 bookmarks=true,bookmarksnumbered=false,bookmarksopen=false,
 breaklinks=false,pdfborder={0 0 1},backref=false,colorlinks=true]
 {hyperref}
\hypersetup{
 allcolors=blue}
\usepackage{breakurl}

\makeatletter

\floatstyle{ruled}
\newfloat{algorithm}{tbp}{loa}
\providecommand{\algorithmname}{Algorithm}
\floatname{algorithm}{\protect\algorithmname}

\numberwithin{equation}{section}
\numberwithin{figure}{section}
\theoremstyle{plain}
\newtheorem{thm}{\protect\theoremname}
  \theoremstyle{plain}
  \newtheorem{lem}[thm]{\protect\lemmaname}
  \theoremstyle{plain}
  \newtheorem{prop}[thm]{\protect\propositionname}

\usepackage{graphicx}
\allowdisplaybreaks[4]
\usepackage[a4paper]{geometry} 
\let\tempone\itemize
\let\temptwo\enditemize
\renewenvironment{itemize}{\tempone\addtolength{\itemsep}{0.3\baselineskip}}{\temptwo}

\usepackage{listings}
\usepackage{xcolor}
\makeatother

\usepackage{listings}
  \providecommand{\lemmaname}{Lemma}
  \providecommand{\propositionname}{Proposition}
\providecommand{\theoremname}{Theorem}

\begin{document}

\title{Permutation polynomials of degree $8$ over finite fields of odd characteristic}

\author{Xiang Fan}

\email{fanx8@mail.sysu.edu.cn}

\address{School of Mathematics, Sun Yat-sen University, Guangzhou 510275, China}
\begin{abstract}
This paper provides an algorithmic generalization of Dickson's method
of classifying permutation polynomials (PPs) of a given degree $d$
over finite fields. Dickson's idea is to formulate from Hermite's
criterion several polynomial equations satisfied by the coefficients
of an arbitrary PP of degree $d$. Previous classifications of PPs
of degree at most $6$ were essentially deduced from manual analysis
of these polynomial equations. However, these polynomials, needed
for that purpose when $d>6$, are too complicated to solve. Our idea
is to make them more solvable by calculating some radicals of ideals
generated by them, implemented by a computer algebra system (CAS).
Our algorithms running in SageMath 8.6 on a personal computer work
very fast to determine all PPs of degree $8$ over an arbitrary finite
field of odd order $q>8$. The main result is that for an odd prime
power $q>8$, a PP $f$ of degree $8$ exists over the finite field
of order $q$ if and only if $q\leqslant 31$ and $q\not\equiv 1\ (\mathrm{mod}\ 8)$, 
and $f$ is explicitly listed up to linear transformations.
\end{abstract}



\maketitle

\section{\label{sec:Intr} Introduction}

Denote by $\mathbb{F}_{q}$ the finite field of order $q$, and write
$\mathbb{F}_{q}^{*}=\mathbb{F}_{q}\backslash\{0\}$. Reserve the letter
$x$ for the variable of the univariate polynomial ring $\mathbb{F}_{q}[x]$
over $\mathbb{F}_{q}$. An arbitrary map from $\mathbb{F}_{q}$ to
itself can be represented as $(a\in\mathbb{F}_{q}\mapsto f(a))$ by
a polynomial $f$ in $\mathbb{F}_{q}[x]$. We call $f$ is a \emph{permutation
polynomial} (PP) over $\mathbb{F}_{q}$ if it represents a permutation
of $\mathbb{F}_{q}$.

Rooted in Hermite \citep{Hermite1863sur} and Dickson \citep{Dickson1897analytic}
in the nineteenth century, the study on PPs over finite fields has
aroused a growing interest, partially due to its valuable applications
in other areas of mathematics and engineering, such as cryptography,
coding theory, combinatorial designs and so on. For example, a special
class of PPs called Dickson polynomials (introduced by \citep{Dickson1897analytic})
played a key role in Ding and Yuan's breakthrough construction \citep{DingYuan2006family}
of a new family of skew Hadamard difference sets in combinatorics.

Although dozens of classes of PPs (with good appearance or properties)
have been found (see \citep{Mullen2013Handbook,Hou2015survey} for
recent surveys), the basic problem of classification of PPs of prescribed
forms is still challenging. In his pioneering thesis work \citep{Dickson1897analytic}
on PPs, L.~E.~Dickson discussed the classification of all PPs of
a given degree $d$ over an arbitrary finite field $\mathbb{F}_{q}$.
Replacing PPs by their reductions modulo $x^{q}-x$ if necessary,
it is assumed that $d<q$. Up-to-date results on this classification
are as follows:
\begin{itemize}
\item by Dickson's 1896 thesis \citep{Dickson1897analytic} for $d\leqslant5$
with any $q$, and for $d=6$ with any odd $q$;
\item by Li, Chandler and Xiang \citep{LiChandlerXiang2010permutation}
in 2010 for $d=6$ or $7$ with any even $q$;
\item by the author\textquoteright s recent \citep{Fan2019PP7} for $d=7$
with any odd $q$, and \citep{Fan2019PP8p2} for $d=8$ with any even
$q$.
\end{itemize}
The present paper contributes to this line by classifying all PPs
of degree $8$ over an arbitrary $\mathbb{F}_{q}$ of odd order $q>8$.
More generally, we actually provide an algorithmic generalization
of Dickson's method of classifying PPs of a given degree $d$ over
finite fields. Dickson's idea is to formulate from Hermite's criterion
several polynomial equations satisfied by the coefficients of an arbitrary
PP of degree $d$. Previously known classifications of PPs of degree
at most $6$ were essentially deduced from manual analysis of these
polynomial equations. However, these polynomials, needed for that
purpose when $d>6$, are too complicated to solve. Our idea is to
make them more solvable by calculating some  radicals of ideals generated
by them, implemented by a computer algebra system (CAS). Our algorithms
running in SageMath 8.6 on a personal computer work very fast to determine
all PPs of degree $8$ over finite fields of odd order, as described
below.
\begin{thm}
For an odd prime power $q>8$, PPs of degree $8$ exist over $\mathbb{F}_{q}$
if and only if
\[
q\in\{11,13,19,23,27,29,31\}.
\]
Moreover, PPs of degree $8$ in normalized form over finite fields
are explicitly listed in Propositions \ref{prop:31}, \ref{prop:29},
\ref{prop:27}, \ref{prop:23}, \ref{prop:19}, \ref{prop:13} and
\ref{prop:11}.
\end{thm}
It is worth to mention that all previous classifications of PPs of
degree at most $7$ can be recovered very quickly by similar algorithms
in our approach here, with calculations implemented by a personal
computer. This approach is different from that in the author's \citep{Fan2019PP7}
classifying PPs of degree $7$. Roughly speaking, \citep{Fan2019PP7}
uses only two simple equations provided by Hermite's criterion, and
its main algorithm is a brute-force search (though optimized by linear
transformations), which cannot work out for degree $8$ and $q>100$
in an acceptable period of time. On the contrary, the approach here
will also work for degree a little larger than $8$ under the support
of a personal computer. We have already done some computations for
degree $9$, and is writing the results in a forthcoming paper.

The structure of this paper is as follows. Section \ref{sec:HC} establishes
Algorithm \ref{alg:HC8} for explicit polynomial equations on coefficients
of PPs of degree $8$ by Hermite\textquoteright s criterion. Section
\ref{sec:Non-exist} verifies the non-existence of PPs of degree $8$
over finite fields of odd order $q>31$, by calculations of some radicals
of ideals generated by polynomials provided by Algorithm \ref{alg:HC8}.
Section \ref{sec:Results} explicitly lists all PPs of degree $8$
in normalized form over $\mathbb{F}_{q}$ of odd order $q$ such that
$8<q\leqslant31$, by a brute-force search.

\section{\label{sec:HC}Hermite's criterion}

The main tool employed by this paper (and by the above mentioned works)
is \emph{Hermite\textquoteright s criterion} for PPs over finite fields
on their coefficients. Introduced by Dickson \citep{Dickson1897analytic}
as a generalization of its prime field case in Hermite \citep{Hermite1863sur},
this criterion is usually named after Hermite, but also sometimes
called Hermite-Dickson criterion. We state here an explicit version
of it from \citep{LidlNiederreiter1997book}, with the following notations
specified.

Let $\mathbb{N}=\{n\in\mathbb{Z}:n\geqslant0\}$. For $n\in\mathbb{N}$
and $f\in\mathbb{F}_{q}[x]$, denote by $[x^{n}:f]$ the coefficient
of $x^{n}$ in $f(x)$. In other words, for a nonzero $f\in\mathbb{F}_{q}[x]$,
we have $f(x)=\sum_{n=0}^{\deg(f)}[x^{n}:f]\cdot x^{n}$, where $\deg(f)=\max\{n\in\mathbb{N}:[x^{n}:f]\neq0\}$.
For $t\in\mathbb{R}$, let $\lfloor t\rfloor$ indicate the largest
integer $\leqslant t$.
\begin{lem}
[{Hermite\textquoteright s criterion \citep[Theorem 7.6]{LidlNiederreiter1997book}}]\label{lem:HC}
Let $f\in\mathbb{F}_{q}[x]$. A necessary and sufficient condition
for $f$ to be a PP over $\mathbb{F}_{q}$ is that 
\[
\sum_{w=1}^{\lfloor\frac{\deg(f^{m})}{q-1}\rfloor}[x^{w(q-1)}:f^{m}]\begin{cases}
=0 & \text{for }1\leqslant m\leqslant q-2,\\
\neq0 & \text{for }m=q-1.
\end{cases}
\]

\end{lem}
Let us show how to calculate $[x^{n}:f^{m}]$ explicitly via multinomial
coefficients. Consider a polynomial $f$ of degree $d$ in $\mathbb{F}_{q}[x]$.
Suppose $\gcd(d,q)=1$ (noting that we aim for $d=8$ with an odd
$q$). By linear transformations, we may assume that $f$ is in normalized
form, i.e. $f(x)=x^{d}+\sum_{i=1}^{d-2}a_{i}x^{i}$ with all $a_{i}\in\mathbb{F}_{q}$.
For integers $j$, $j_{1}$, $j_{2}$, $\dots$, and $j_{d}$, define
the associated \emph{multinomial coefficient} as 
\[
\binom{j}{j_{1},j_{2},\dots,j_{d}}:=\begin{cases}
\dfrac{j!}{j_{1}!j_{2}!\cdots j_{d}!} & \text{if }j=j_{1}+j_{2}+\cdots+j_{d}\text{ and all }j_{1},\dots,j_{d}\geqslant0,\\
0 & \text{otherwise}.
\end{cases}
\]
By the multinomial theorem, 
\[
f(x)^{m}=\sum_{\sum_{i=1}^{d-2}j_{i}+j_{d}=m}\binom{m}{j_{1},j_{2},\dots,j_{d-2},j_{d}}(\prod_{i=1}^{d-2}a_{i}^{j_{i}})\cdot x^{\sum_{i=1}^{d-2}ij_{i}+dj_{d}}.
\]
Therefore, 
\begin{align*}
[x^{n}:f(x)^{m}] & =\sum_{\substack{\sum_{i=1}^{d-2}j_{i}+j_{d}=m\\
\sum_{i=1}^{d-2}ij_{i}+dj_{d}=n.
}
}\binom{m}{j_{1},j_{2},\dots,j_{d-2},j_{d}}\prod_{i=1}^{d-2}a_{i}^{j_{i}}\\
 & =\sum_{\sum_{i=1}^{d-2}(d-i)j_{i}=dm-n}\binom{m}{j_{1},j_{2},\dots,j_{d-2},m-\sum_{i=1}^{d-2}j_{i}}\prod_{i=1}^{d-2}a_{i}^{j_{i}}.
\end{align*}
Define a multivariate polynomial $\mathbf{HC}_{d}(q,m)$ in
$\mathbb{F}_{q}[x_{1},x_{2},\dots,x_{d-2}]$ (with $d-2$ variables
$x_{i}$) as
\[
\mathbf{HC}_{d}(q,m):=\sum_{w=1}^{\lfloor\frac{dm}{q-1}\rfloor}\sum_{\sum_{i=1}^{d-2}(d-i)j_{i}=dm-w(q-1)}\binom{m}{j_{1},j_{2},\dots,j_{d-2},m-\sum_{i=1}^{d-2}j_{i}}\prod_{i=1}^{d-2}x_{i}^{j_{i}}.
\]
Then Hermite\textquoteright s criterion asserts that $f(x)=x^{d}+\sum_{i=1}^{d-2}a_{i}x^{i}$
(with all $a_{i}\in\mathbb{F}_{q}$) is a PP over $\mathbb{F}_{q}$
if and only if 
\[
\mathbf{HC}_{d}(q,m)(a_{1},a_{2},\dots,a_{d-2})\begin{cases}
=0 & \text{for }1\leqslant m\leqslant q-2,\\
\neq0 & \text{for }m=q-1.
\end{cases}
\]

\begin{itemize}
\item When $q\equiv1\ (\mathrm{mod}\ d)$, no PP of degree $d$ exists over
$\mathbb{F}_{q}$ because $\mathbf{HC}_{d}(q,\frac{q-1}{d})=1$.
\item If $q\not\equiv0,1\ (\mathrm{mod}\ d)$, then $\mathbf{HC}_{d}(q,m)=0$
when $m\leqslant\lfloor\frac{q}{d}\rfloor=\lfloor\frac{q-1}{d}\rfloor<\frac{q-1}{d}$.
\end{itemize}

When $\gcd(d,q)=1$, previous classifications of PPs of degree $d\leqslant6$
were essentially deduced from manual analysis of the polynomial equations
$\mathbf{HC}_{d}(q,m)(a_{1},a_{2},\dots,a_{d-2})=0$ on $a_{i}$'s
provided by Hermite's criterion. Roughly speaking, $d-2$ equations
$\mathbf{HC}_{d}(q,m)(a_{1},a_{2},\dots,a_{d-2})=0$ for $\lfloor\frac{q}{d}\rfloor+1\leqslant m\leqslant\lfloor\frac{q}{d}\rfloor+d-2$
are enough to determine $(a_{1},a_{2},\dots,a_{d-2})\in\mathbb{F}_{q}^{d-2}$
when $q>d(d-2)$. However, when $d>6$, these $d-2$ polynomials $\mathbf{HC}_{d}(q,m)$
are too long to write down explicitly, let alone to solve. Our main
idea is to show the non-existence of $(a_{1},\dots,a_{d-2})$ for
$q>d(d-2)$, by calculating some radicals of ideals generated by some
$\mathbf{HC}_{d}(q,m)$, implemented by a computer algebra system
(CAS) running on a personal computer. This idea works at least for
$d\leqslant8$, provided that $\gcd(d,q)=1$.

In practice, all algorithms of this paper runs in SageMath \citep{SageMath}
(version 8.6), a free open-source CAS combining the power of many
existing open-source packages, such as NumPy, SciPy, Sympy, Maxima,
R, GAP, \textsc{Singular} and many more, into a common Python-based
interface. SageMath uses a Python-like language, which is very readable
even for those without programming experience.

For $d=8$, the multivariate polynomial $\mathbf{HC}_{8}(q,m)$ in
$\mathbb{F}_{q}[x_{1},x_{2},\dots,x_{6}]$ is defined as 
\[
\mathbf{HC}_{8}(q,m):=\sum_{\substack{7j_{1}+6j_{2}+5j_{3}+4j_{4}+3j_{5}+2j_{6}=8m-n\\
n\in\{w(q-1):\ 1\leqslant w\leqslant\lfloor\frac{8m}{q-1}\rfloor\}
}
}\binom{m}{j_{1},j_{2},\dots,j_{6},m-\sum_{i=1}^{6}j_{i}}\prod_{i=1}^{6}x_{i}^{j_{i}}.
\]
Algorithm \ref{alg:HC8} realizes $\mathbf{HC}_{8}(q,m)$ as a SageMath
function $\mathtt{HC8}(q,m)$, outputting a multivariate polynomial
in $\mathbb{F}_{q}[x_{1},x_{2},\dots,x_{6}]$.

\begin{algorithm}[h]
\protect\caption{\label{alg:HC8}To calculate $\mathbf{HC}_{8}(q,m)$ in $\mathbb{F}_{q}[x_{1},x_{2},\dots,x_{6}]$}

\begin{lstlisting}
def HC8(q,m):
    HC = 0; K.<x1,x2,x3,x4,x5,x6> = PolynomialRing(GF(q))
    for n in range(q-1,8*m+1,q-1):
        u = 8*m-n
        for j1 in range(u//7+1):
            for j2 in range((u-7*j1)//6+1):
                for j3 in range((u-7*j1-6*j2)//5+1):
                    for j4 in range((u-7*j1-6*j2-5*j3)//4+1):
                        for j5 in range((u-7*j1-6*j2-5*j3-4*j4)//3+1):
                            v = u-7*j1-6*j2-5*j3-4*j4-3*j5
                            if is_odd(v): continue
                            j6 = v//2; j = j1+j2+j3+j4+j5+j6
                            if j>m: continue
                            c = int(multinomial(j1,j2,j3,j4,j5,j6,m-j))
                            HC += c*x1^j1*x2^j2*x3^j3*x4^j4*x5^j5*x6^j6
    return HC
\end{lstlisting}
\end{algorithm}

\section{\label{sec:Non-exist}Non-existence for odd $q>31$}

In an address before the Mathematical Association of America in 1966,
L.~Carlitz conjectured a constant $C_{n}$ for each positive even
integer $n$ such that no PP of degree $n$ exists over $\mathbb{F}_{q}$
of odd order $q>C_{n}$. If $\gcd(n,q)=1$, it is verified by Hayes
\citep{Hayes1967geometric} with a stronger result as follows.
\begin{lem}
\label{lem:Carlitz-Wan} \textup{\citep[Theorem 3.4]{Hayes1967geometric}}
Given a positive integer $n$, there is a constant $C_{n}$ (depending
only on $n$) such that for any prime power $q>C_{n}$ with $\gcd(n,q)=1$,
a PP of degree $n$ exists over $\mathbb{F}_{q}$ only if $\gcd(n,q-1)=1$.
\end{lem}
Lemma \ref{lem:Carlitz-Wan} without the assumption $\gcd(n,q)=1$
is called the Carlitz\textendash Wan conjecture, which is now a theorem
by \citep{FriedGuralnickSaxl1993Schur,CohenFried1995Lenstra}. For
Lemma $2$ (and the Carlitz\textendash Wan conjecture) to hold, $C_{n}$
can be taken as $n^{4}$ by von zur Gathen \citep{Gathen1991values},
as $n^{2}(n-2)^{2}$ by Chahal and Ghorpade \citep{ChahalGhorpade2018Carlitz},
and as 
\[
\left\lfloor \left(\dfrac{(n-2)(n-3)+\sqrt{(n-2)^{2}(n-3)^{2}+8n-12}}{2}\right)^{2}\right\rfloor 
\]
by the author's preprint \citep{Fan2019Weil}. Especially, when $n=8$,
the above expression is $925$. The greatest prime power below $925$
is $919$. So $C_{8}$ can be taken as $919$.

This section will further refine the bound $C_{8}$ to $31$ for the
original version of Carlitz conjecture: no PP of degree $8$ exists
over $\mathbb{F}_{q}$ of odd order $q>31$. The main method is to
calculate some radicals of ideals generated by some $\mathbf{HC}_{8}(q,m)$,
with details as follows.

Consider a PP $f\in\mathbb{F}_{q}[x]$ of degree $8$ over $\mathbb{F}_{q}$
of odd order $q=8t+s$ with $1\leqslant t\in\mathbb{Z}$ and $s\in\{3,5,7\}$.
Note that $s\neq1$ by Hermite's criterion. Without loss of loss of
generality, assume that $f$ is in normalized form, i.e. $f(x)=x^{8}+\sum_{i=1}^{6}a_{i}x^{i}$
with all $a_{i}\in\mathbb{F}_{q}$. Hermite's criterion ensures that
$(a_{1},a_{2},\dots,a_{6})\in\mathbb{F}_{q}^{6}$ is a vanishing point
of every polynomial $\mathbf{HC}_{8}(q,m)$ such that $1\leqslant m\leqslant q-2$,
and thus of every polynomial in the radical of any ideal in $\mathbb{F}_{q}[x_{1},x_{2},\dots,x_{6}]$
generated by some of them. Here the \emph{radical} $\sqrt{I}$ of
an ideal $I$ in a ring $R$ is defined as 
\[
\sqrt{I}:=\{g\in R:g^{m}\in I\text{ for some positive }m\in\mathbb{Z}\}.
\]
Especially, for $1\leqslant k\in\mathbb{Z}$, the radical $\mathbf{Rad}_{8}(q,k)$
of the ideal generated by $k$ polynomials $\mathbf{HC}_{8}(q,m)$
with $\lfloor\frac{q}{8}\rfloor+1\leqslant m\leqslant\lfloor\frac{q}{8}\rfloor+k$
can be calculated by SageMath function $\mathtt{Rad8}(q,k)$ in Algorithm
\ref{alg:Rad8}.

\begin{algorithm}[h]
\protect\caption{\label{alg:Rad8}To calculate the radical $\mathbf{Rad}_{8}(q,k)$}

\begin{lstlisting}
def Rad8(q,k):
    return Ideal([HC8(q,q//8+1+i) for i in range(k)]).radical()
\end{lstlisting}
\end{algorithm}

SageMath uses \textsc{Singular} \citep{Singular} to implement the
calculation for radicals of ideals in multivariate polynomial rings
over fields, based on the algorithm of Kempers \citep{Kemper2002calculation}
in positive characteristic.

We pick out the following $\mathbf{Rad}_{8}(q,k)$ for $q>31$ given
by Algorithm \ref{alg:Rad8} in SageMath 8.6. Each output is of the
form $\mathtt{Ideal}(g_{1},g_{2},\dots,g_{s})$, denoting the ideal
generated by $g_{1},g_{2},\dots,g_{s}$ in $\mathbb{F}_{q}[x_{1},x_{2},\dots,x_{6}]$.
By definition, every $g_{i}$ in the output vanishes at $(a_{1},a_{2},\dots,a_{6})\in\mathbb{F}_{q}^{6}$
for any PP of the form $f(x)=x^{8}+\sum_{i=1}^{6}a_{i}x^{i}$ over
$\mathbb{F}_{q}$. For each $q$, we choose a suitable $k$ to manufacture
enough good $g_{i}$'s for our purpose. Our choice of $k$ might be
not as small as possible, but makes the running time of $\mathtt{Rad8}(q,k)$
for the same result as short as possible.
\begin{itemize}
\item $\mathbf{Rad}_{8}(37,8)=\mathtt{Ideal}(1)$. So no PP of degree $8$
exists over $\mathbb{F}_{37}$.
\item $\mathbf{Rad}_{8}(43,7)=\mathtt{Ideal}(x_{6},x_{5},x_{4},x_{3,}x_{2},x_{1})$.
Then $f(x)=x^{8}$ is clearly not a PP over $\mathbb{F}_{43}$. So
no PP of degree $8$ exists over $\mathbb{F}_{43}$.
\item $\mathbf{Rad}_{8}(47,7)=\mathtt{Ideal}(x_{6},x_{5},x_{4},x_{3,}x_{2},x_{1})$.
So no PP of degree $8$ exists over $\mathbb{F}_{47}$.
\item $\mathbf{Rad}_{8}(53,7)=\mathtt{Ideal}(1)$. So no PP of degree $8$
exists over $\mathbb{F}_{53}$. 
\end{itemize}
The above calculations indicate that no PP of degree $8$ exists over
$\mathbb{F}_{q}$ of odd order $q$ if $31<q\leqslant53$.
\begin{itemize}
\item For any prime power $q\equiv7\ (\mathrm{mod}\ 8)$ with $53<q\leqslant919$,
\[
\mathbf{Rad}_{8}(q,7)=\mathtt{Ideal}(x_{6},x_{5},x_{3,}x_{2},x_{1}).
\]
Then $f(x)=x^{8}+a_{4}x^{4}$. Now $q=8t+7$ with $1\leqslant t\in\mathbb{Z}$.
For $m=\lfloor\frac{q}{4}\rfloor+1=2t+2<q-1$, by Hermite\textquoteright s
criterion, 
\begin{align*}
0 & =\mathbf{HC}_{8}(q,m)(0,0,0,a_{4},0,0)=\sum_{\substack{4j_{4}=16t+16-n\\
n\in\{w(8t+6):\ 1\leqslant w\leqslant\lfloor\frac{16t+16}{8t+6}\rfloor\}
}
}\binom{m}{j_{4},m-j_{4}}a_{4}^{j_{4}}\\
 & =\sum_{4j_{4}=16t+16-2(8t+6)}\binom{m}{j_{4},m-j_{4}}a_{4}^{j_{4}}=ma_{4}\in\mathbb{F}_{q}
\end{align*}
Note that $\gcd(m,q)=1$ as $q=4m-1$, thus $a_{4}=0$. Then $f(x)=x^{8}$
is clearly not a PP over $\mathbb{F}_{q}$. So no PP of degree $8$
exists over $\mathbb{F}_{q}$ if $q\equiv7\ (\mathrm{mod}\ 8)$ and
$53<q\leqslant919$.
\item For any prime $q\equiv3\ (\mathrm{mod}\ 8)$ with $53<q\leqslant919$,
\[
\mathbf{Rad}_{8}(q,7)=\mathtt{Ideal}(x_{5},x_{3},x_{1},x_{4}x_{6}-10x_{2},x_{6}^{3}-32x_{2}).
\]

\item For any prime $q\equiv5\ (\mathrm{mod}\ 8)$ with $53<q\leqslant919$,
\[
\mathbf{Rad}_{8}(q,7)=\mathtt{Ideal}(x_{5},x_{3},x_{1},x_{6}^{2}+\alpha(q)x_{4},x_{4}x_{6}-10x_{2}),
\]
where $\alpha(q)\in\mathbb{F}_{q}$ depends on $q$.
\item For any non-prime prime power $q\equiv3$ or $5\ (\mathrm{mod}\ 8)$
with $53<q\leqslant919$, i.e. $q=5^{3}$ or $3^{5}$, 
\begin{align*}
\mathbf{Rad}_{8}(5^{3},9) & =\mathtt{Ideal}(x_{5},x_{4},x_{3},x_{1},x_{6}^{3}-2x_{2}),\\
\mathbf{Rad}_{8}(3^{5},19) & =\mathtt{Ideal}(x_{5},x_{3},x_{1},x_{4}x_{6}-x_{2},x_{6}^{3}+x_{2},x_{2}x_{6}^{2}+x_{2}x_{4},x_{2}x_{4}^{2}+x_{2}^{2}x_{6}).
\end{align*}

\end{itemize}
For any prime power $q\equiv3$ or $5$ $(\mathrm{mod}\ 8)$ with
$53<q\leqslant919$, the above calculations imply that $a_{1}=a_{3}=a_{5}=0$
by Hermite\textquoteright s criterion. Moreover, $\mathbf{HC}_{8}(q,m)(0,x_{2},0,x_{4},0,x_{6})$
can be calculated by the SageMath function $\mathtt{HC8new}(q,m)$
in Algorithm \ref{alg:HC8new}.

\begin{algorithm}[h]
\protect\caption{\label{alg:HC8new}To calculate $\mathbf{HC}_{8}(q,m)(0,x_{2},0,x_{4},0,x_{6})$}

\begin{lstlisting}
def HC8new(q,m):
    HC = 0; K.<x1,x2,x3,x4,x5,x6> = PolynomialRing(GF(q))
    for n in range(q-1,8*m+1,q-1):
        u = 8*m-n
        for j2 in range(u//6+1):
            for j4 in range((u-6*j2)//4+1):
                v = u-6*j2-4*j4
                if is_odd(v): continue
                j6 = v//2; j = j2+j4+j6
                if j<=m:
                    HC += int(multinomial(j2,j4,j6,m-j))*x2^j2*x4^j4*x6^j6
    return HC
\end{lstlisting}

\end{algorithm}

Inspired by the case of $q\equiv7\ (\mathrm{mod}\ 8)$, we try to
run $\mathtt{HC8new}(q,m)$ with $m=\lfloor\frac{q}{4}\rfloor+1$.
After some experiments, we fortunately see that the output of the
following SageMath codes in Algorithm \ref{alg:non-ex}, which prints
$\mathtt{Ideal}(x_{6},x_{5},x_{4},x_{3,}x_{2},x_{1})$ for every prime
power $q\equiv3$ or $5$ $(\mathrm{mod}\ 8)$ with $53<q\leqslant919$,
verifies the non-existence of PPs of degree $8$ over $\mathbb{F}_{q}$
for these $q$.

\begin{algorithm}[h]
\protect\caption{\label{alg:non-ex}Non-existence of PPs of degree $8$ over $\mathbb{F}_{q}$
if $q\equiv3$ or $5$ $(\mathrm{mod}\ 8)$ with $53<q\leqslant919$}

\begin{lstlisting}
for q in prime_range(54,920)+[5^3,3^5]:
    if q%8 in [1,7]: continue
    K.<x1,x2,x3,x4,x5,x6> = PolynomialRing(GF(q))
    if q == 5^3: k = 9
    elif q== 3^5: k = 19
    else: k = 7
    R = Rad8(q,k).gens()
    if x1 in R and x3 in R and x5 in R:
        print(Ideal(R+[HC8new(q,q//4+1)]).radical())
    else: print("Fail when q = %d" % q)
\end{lstlisting}

\end{algorithm}

Calculations in this section for odd prime powers $q$ such that $31<q\leqslant919$,
together with Lemma \ref{lem:Carlitz-Wan} in which $C_{8}=919$ by
\citep{Fan2019Weil}, ensure the following Theorem \ref{thm:Non-ex}
of non-existence.
\begin{thm}
\label{thm:Non-ex}No PP of degree $8$ exists over any finite field
$\mathbb{F}_{q}$ of odd order $q>31$.
\end{thm}

\section{\label{sec:Results}Explicit results}

This section lists all PPs of degree $8$ (in normalized form) over
$\mathbb{F}_{q}$ of odd order $q>8$. Noting that $q\not\equiv1\ (\mathrm{mod}\ 8)$
by Hermite\textquoteright s criterion, and that $q\leqslant31$ by
Theorem \ref{thm:Non-ex}, we indeed have 
\[
q\in\{11,13,19,23,27,29,31\}.
\]

To make the resulting list as short as possible, we first investigate
the linear transformation relations among polynomials of degree $8$.
Two polynomials $f$ and $g$ in $\mathbb{F}_{q}[x]$ are said to
be \emph{related by linear transformations} (\emph{linearly related}
for short) if there exist $s,t\in\mathbb{F}_{q}^{*}$ and $u,v\in\mathbb{F}_{q}$
such that $g(x)=sf(tx+u)+v$. Note that linearly related $f$ and
$g$ possess the same degree, and $f$ is a PP over $\mathbb{F}_{q}$
if and only if so is $g$.
\begin{prop}
\label{prop:Eq8} Let $q$ be an odd prime power. Then each polynomial
of degree $8$ in $\mathbb{F}_{q}[x]$ is linearly related to some
$f\in\mathbb{F}_{q}[x]$ in normalized form, namely $f(x)=x^{8}+\sum_{i=1}^{6}a_{i}x^{i}$
with all $a_{i}\in\mathbb{F}_{q}$. Moreover, for another $g(x)=x^{8}+\sum_{i=1}^{6}b_{i}x^{i}\in\mathbb{F}_{q}[x]$
with all $b_{i}\in\mathbb{F}_{q}$, we have that $f$ and $g$ are
linearly related if and only if $f(x)=t^{8}g(t^{-1}x)$ for some $t\in\mathbb{F}_{q}^{*}$,
i.e. 
\[
(a_{6},a_{5},a_{4},a_{3},a_{2},a_{1})=(t^{2}b_{6},t^{3}b_{5},t^{4}b_{4},t^{5}b_{3},t^{6}b_{2},t^{7}b_{1}).
\]
\end{prop}
\begin{proof}
Each polynomial $h$ of degree $8$ in $\mathbb{F}_{q}[x]$ can be
written as $h(x)=\sum_{i=1}^{8}c_{i}x^{i}$, with all $c_{i}\in\mathbb{F}_{q}$
and $c_{8}\neq0$. Let $f(x)=c_{8}^{-1}h(x-8^{-1}c_{8}^{-1}c_{7})-c_{8}^{-1}h(-8^{-1}c_{8}^{-1}c_{7})$,
then $f$ is in normalized form and linearly related to $h$.

Suppose that $f(x)=x^{8}+\sum_{i=1}^{6}a_{i}x^{i}$ and $g(x)=x^{8}+\sum_{i=1}^{6}b_{i}x^{i}$
(with all $a_{i}$, $b_{i}\in\mathbb{F}_{q}$) are linearly related,
namely $g(x)=sf(tx+u)+v$ with $s,t\in\mathbb{F}_{q}^{*}$ and $u,v\in\mathbb{F}_{q}$.
Clearly, $st^{8}=1$ and $8st^{7}u=0$, considering the coefficients
of $x^{8}$ and $x^{7}$ respectively. So $s=t^{-8}$ and $u=0$.
Then $g(x)=t^{-8}f(tx)+v$, where $v=g(0)-t^{-8}f(0)=0$.
\end{proof}
Hereafter, we consider a polynomial in $\mathbb{F}_{q}[x]$ in normalized
form $f(x)=x^{8}+\sum_{i=1}^{6}a_{i}x^{i}$ with all $a_{i}\in\mathbb{F}_{q}$.
Algorithm \ref{alg:PP8} defines a SageMath function $\mathtt{isPP8}(q,a_{6},a_{5},a_{4},a_{3},a_{2},a_{1})$
to examine whether or not $f$ is a PP over $\mathbb{F}_{q}$. By
Wan \citep{Wan1993padic}, it suffices to test whether or not the
value set $\{f(c):c\in\mathbb{F}_{q}\}$ contains $\lfloor q-\frac{q-1}{8}\rfloor+1$
distinct values. In the following subsections, we will look for all
points $(a_{1},a_{2},\dots,a_{6})$ in $\mathbb{F}_{q}^{6}$ for $f$
to be a PP over $\mathbb{F}_{q}$, up to linear transformations as
indicated by Proposition \ref{prop:Eq8}, for any odd prime power
$q\not\equiv1\ (\mathrm{mod}\ 8)$ with $8<q\leqslant31$, through
a brute-force search optimized by a case-by-case analysis.

\begin{algorithm}[h]

\protect\caption{\label{alg:PP8}To examine whether $f(x)=x^{8}+\sum_{i=1}^{6}a_{i}x^{i}$
is a PP over $\mathbb{F}_{q}$}

\begin{lstlisting}
def isPP8(q,a6,a5,a4,a3,a2,a1):
    V = []
    for x in list(GF(q,'e'))[0:1+int(q-(q-1)/8)]:
        t = x^8+a6*x^6+a5*x^5+a4*x^4+a3*x^3+a2*x^2+a1*x
        if t in V: return False
        else: V.append(t)
    return True
\end{lstlisting}
\end{algorithm}

\subsection{Case $q=31$}

Run the following SageMath codes for polynomials in $\mathbf{Rad}_{8}(31,12)$
with at most two terms.

\begin{lstlisting}
for g in Rad8(31,12).gens():
    if g.number_of_terms()<3: print g
\end{lstlisting}

The output prints $x_{6}$, $x_{2}^{5}-1$, $x_{3}^{6}-1$, $x_{4}^{15}-1$
and $x_{1}^{30}-1$, all of which vanish at $(a_{1},a_{2},\dots,a_{6})\in\mathbb{F}_{31}^{6}$.
So $a_{6}=0\neq a_{1}$, and $a_{2}^{5}=a_{3}^{6}=a_{4}^{15}=1$.

Note that $f(x)$ is linearly related to 
\[
t^{8}f(t^{-1}x)=x^{8}+t^{2}a_{6}x^{6}+t^{3}a_{5}x^{5}+t^{4}a_{4}x^{4}+t^{5}a_{3}x^{3}+t^{6}a_{2}x^{2}+t^{7}a_{1}x.
\]
As $q-1=30$ is coprime to $7$, $a_{1}=a^{7}$ for some $a\in\mathbb{F}_{31}^{*}$.
Replacing $f(x)$ by $a^{-8}f(ax)$ if necessary, we assume $a_{1}=1$
without loss of generality. So Algorithm \ref{alg:31} lists PPs of
degree $8$ in normalized form over $\mathbb{F}_{31}$, up to linear
transformations.

\begin{algorithm}[h]
\protect\caption{\label{alg:31} To list PPs of degree $8$ in normalized form over
$\mathbb{F}_{31}$}

\begin{lstlisting}
q = 31; F = GF(q)
for a5 in F:
    for a4 in [a4 for a4 in F if a4^15==1]:
        for a3 in [a3 for a3 in F if a3^6==1]:
            for a2 in [a2 for a2 in F if a2^5==1]:
                if isPP8(q,0,a5,a4,a3,a2,1): print(0,a5,a4,a3,a2,1)
\end{lstlisting}

\end{algorithm}

The output of Algorithm \ref{alg:31} is $(0,19,25,6,2,1)$, which
gives the following Proposition \ref{prop:31}.
\begin{prop}
\label{prop:31}All PPs of degree $8$ in normalized form over $\mathbb{F}_{31}$
are exactly 
\[
x^{8}+19t^{3}x^{5}+25t^{4}x^{4}+6t^{5}x^{3}+2t^{6}x^{2}+t^{7}x,
\]
with $t$ running through $\mathbb{F}_{31}^{*}$.
\end{prop}

\subsection{Case $q=29$}

Run the following SageMath codes for polynomials in $\mathbf{Rad}_{8}(29,7)$
with at most three terms.

\begin{lstlisting}
for g in Rad8(29,7).gens():
    if g.number_of_terms()<4: print g
\end{lstlisting}

The output prints $x_{6}^{2}-9x_{4}$, $x_{1}^{8}+4x_{1}^{4}-5$,
$x_{2}^{15}-x_{2}$ and $x_{3}^{29}-x_{3}$, all of which vanish at
$(a_{1},a_{2},\dots,a_{6})\in\mathbb{F}_{29}$. So $a_{4}=a_{6}^{2}/9$,
$a_{1}^{8}+4a_{1}^{4}=5$ and $a_{2}^{15}=a_{2}$.

As $q-1=28$ is coprime to $3$, without loss of generality we can
assume $a_{5}\in\{0,1\}$ (by linear transformations if necessary).
So Algorithm \ref{alg:29} lists PPs of degree $8$ in normalized
form over $\mathbb{F}_{29}$, up to linear transformations.

\begin{algorithm}[h]
\protect\caption{\label{alg:29} To list PPs of degree $8$ in normalized form over
$\mathbb{F}_{29}$}

\begin{lstlisting}
q = 29; F = GF(q)
for a6 in F:
    a4 = a6^2/F(9)
    for a5 in [0,1]:
        for a3 in F:
            for a2 in [a2 for a2 in F if a2^15==a2]:
                for a1 in [a1 for a1 in F if a1^8+4*a1^4==5]:
                    if isPP8(q,a6,a5,a4,a3,a2,a1):
                        print(a6,a5,a4,a3,a2,a1)
\end{lstlisting}
\end{algorithm}

The output of Algorithm \ref{alg:29} prints: 
\[
(0,0,0,0,0,4),\ (0,0,0,0,0,10),\ (0,0,0,0,0,19),\ (0,0,0,0,0,25),\ (26,1,1,4,20,1).
\]
Note that $\{4t^{7}:t\in\mathbb{F}_{29}^{*}\}=\{4,10,19,25\}$. So
we get the following Proposition \ref{prop:29}.
\begin{prop}
\label{prop:29}All PPs of degree $8$ in normalized form over $\mathbb{F}_{29}$
are exactly 
\begin{gather*}
x^{8}+4t^{7}x,\qquad x^{8}+26t^{2}x^{6}+t^{3}x^{5}+t^{4}x^{4}+4t^{5}x^{3}+20t^{6}x^{2}+t^{7}x,
\end{gather*}
with $t$ running through $\mathbb{F}_{29}^{*}$.
\end{prop}

\subsection{Case $q=27$}

Note that $(a_{1},a_{3})\neq(0,0)$ by the output of the following
SageMath codes.

\begin{lstlisting}
K.<x1,x2,x3,x4,x5,x6> = PolynomialRing(GF(27))
Ideal([HC8(27,4+i) for i in range(10)]+[x1,x3]).radical()
\end{lstlisting}
The output is $\mathtt{Ideal}(1)$, which indicates that polynomials
$x_{1}$ and $x_{3}$ cannot both vanish at $(a_{1},a_{2},\dots,a_{6})\in\mathbb{F}_{27}$.
So we can make the following assumptions on $(a_{1},a_{2},\dots,a_{6})$,
by linear transformations if necessary.
\begin{itemize}
\item Assume $a_{1}\in\{0,1\}$ as $q-1=26$ is coprime to $7$.
\item When $a_{1}=0$ (and thus $a_{3}\neq0$), assume $a_{3}=1$, as $q-1=26$
is coprime to $5$.
\end{itemize}
Note that $a_{2}=-a_{6}^{3}$ since $\mathbf{HC}_{8}(27,4)=x_{6}^{3}+x_{2}$.
So Algorithm \ref{alg:27} lists PPs of degree $8$ in normalized
form over $\mathbb{F}_{27}$ up to linear transformations, and Proposition
\ref{prop:27} is read off from its output.

\begin{algorithm}[h]
\protect\caption{\label{alg:27} To list PPs of degree $8$ in normalized form over
$\mathbb{F}_{27}$}

\begin{lstlisting}
q = 3^3; F = GF(q,'e')
for (a3,a1) in [(1,0)]+[(a3,1) for a3 in F]:
    for a5 in F:
        for a6 in F:
            a2 = -a6^3
            for a4 in F:
                if isPP8(q,a6,a5,a4,a3,a2,a1): print(a6,a5,a4,a3,a2,a1)
\end{lstlisting}
\end{algorithm}

\begin{prop}
\label{prop:27}All PPs of degree $8$ in normalized form over $\mathbb{F}_{27}$
are exactly those of the form $x^{8}+\sum_{i=1}^{6}t^{8-i}a_{i}x^{i}$,
with $t\in\mathbb{F}_{27}^{*}$ and $(a_{6},a_{5},a_{4},a_{3},a_{2},a_{1})$
listed as follows: 
\begin{alignat*}{3}
 & (1,0,2,1,2,0), & \quad & (1,1,2,1,2,0), & \quad\\
 & (e^{2},2e,2e^{3},e^{10},2e^{6},1), &  & (e^{6},2e^{3},2e^{9},e^{4},e^{5},1), &  & (2e^{5},2e^{9},2e,e^{12},2e^{2},1),\\
 & (e^{2},2e^{4},e^{7},e^{6},2e^{6},1), &  & (e^{6},2e^{12},2e^{8},2e^{5},e^{5},1), &  & (2e^{5},2e^{10},e^{11},e^{2},2e^{2},1),\\
 & (e^{4},2e^{2},e^{7},e^{4},2e^{12},1), &  & (e^{12},2e^{6},2e^{8},e^{12},2e^{10},1), &  & (e^{10},e^{5},e^{11},e^{10},2e^{4},1),\\
 & (2e^{4},e^{10},2e^{12},2e^{11},e^{12},1), &  & (2e^{12},e^{4},2e^{10},2e^{7},e^{10},1), &  & (2e^{10},e^{12},2e^{4},e^{8},e^{4},1).
\end{alignat*}
 
\end{prop}

\subsection{Case $q=23$}

Run the following SageMath codes for polynomials in $\mathbf{Rad}_{8}(23,9)$
with at most four terms.

\begin{lstlisting}
for g in Rad8(23,9).gens():
    if g.number_of_terms()<5: print g
\end{lstlisting}

The output prints $x_{6}$, $x_{4}x_{5}^{2}-11x_{3}^{2}+x_{2}x_{4}+x_{1}x_{5}$,
$x_{5}^{22}-1$, $x_{4}^{22}-1$, $x_{3}^{22}-1$ and $x_{1}^{22}-1$,
all of which vanish at $(a_{1},a_{2},\dots,a_{6})\in\mathbb{F}_{23}$.
So $a_{6}=0\neq a_{j}$ for $j\in\{1,3,4,5\}$, and $a_{4}a_{5}^{2}-11a_{3}^{2}+a_{2}a_{4}+a_{1}a_{5}=0$.

As $q-1=22$ is coprime to $3$, without loss of generality we can
assume $a_{5}=1$ (by linear transformations if necessary). So Algorithm
\ref{alg:23} lists PPs of degree $8$ in normalized form over $\mathbb{F}_{23}$
up to linear transformations, and Proposition \ref{prop:23} is read
off from its output.

\begin{algorithm}[h]
\protect\caption{\label{alg:23} To list PPs of degree $8$ in normalized form over
$\mathbb{F}_{23}$}

\begin{lstlisting}
q = 23; F = GF(q)
for a4 in [a4 for a4 in F if a4!=0]:
    for a3 in [a3 for a3 in F if a3!=0]:
        for a2 in F:
            a1 = -a4+11*a3^2-a2*a4
            if isPP8(q,0,1,a4,a3,a2,a1): print(0,1,a4,a3,a2,a1)
\end{lstlisting}
\end{algorithm}

\begin{prop}
\label{prop:23}All PPs of degree $8$ in normalized form over $\mathbb{F}_{23}$
are exactly those of the form $x^{8}+t^{3}x^{5}+t^{4}a_{4}x^{4}+t^{5}a_{3}x^{3}+t^{6}a_{2}x^{2}+t^{7}a_{1}x$,
with $t\in\mathbb{F}_{23}^{*}$ and $(a_{4},a_{3},a_{2},a_{1})$ listed
as follows: 
\begin{alignat*}{4}
 & (11,1,12,6), & \quad & (11,1,21,22), & \quad & (15,8,16,12), & \quad & (15,16,13,7),\\
 & (16,8,7,1), &  & (19,5,10,20), &  & (20,12,2,6).
\end{alignat*}

\end{prop}

\subsection{Case $q=19$}

Note that $(a_{1},a_{3})\neq(0,0)$ by the output of the following
SageMath codes.

\begin{lstlisting}
K.<x1,x2,x3,x4,x5,x6> = PolynomialRing(GF(19))
Ideal([HC8(19,3+i) for i in range(7)]+[x3,x1]).radical()
\end{lstlisting}
The output is $\mathtt{Ideal}(1)$, which indicates that polynomials
$x_{1}$ and $x_{3}$ cannot both vanish at $(a_{1},a_{2},\dots,a_{6})\in\mathbb{F}_{19}$.
So we can make the following assumptions on $(a_{1},a_{2},\dots,a_{6})$,
by linear transformations if necessary.
\begin{itemize}
\item Assume $a_{1}\in\{0,1\}$ as $q-1=18$ is coprime to $7$.
\item When $a_{1}=0$ (and thus $a_{3}\neq0$), assume $a_{3}=1$, as $q-1=18$
is coprime to $5$.
\end{itemize}
Note that $a_{2}=-a_{6}^{3}/3-a_{5}^{2}-2a_{4}a_{6}$ since 
\[
\mathbf{HC}_{8}(19,3)=x_{6}^{3}+3x_{5}^{2}+6x_{4}x_{6}+3x_{2}.
\]
So Algorithm \ref{alg:19} lists PPs of degree $8$ in normalized
form over $\mathbb{F}_{19}$ up to linear transformations. There are
exactly $48$ linearly related classes of PPs of degree $8$ over
$\mathbb{F}_{19}$ as listed in Proposition \ref{prop:19}, corresponding
to $48$ outputting tuples of Algorithm \ref{alg:19}.

\begin{algorithm}[h]
\protect\caption{\label{alg:19} To list PPs of degree $8$ in normalized form over
$\mathbb{F}_{19}$}

\begin{lstlisting}
q = 19; F = GF(q)
for (a3,a1) in [(1,0)]+[(a3,1) for a3 in F]:
    for a6 in F:
        for a5 in F:
            for a4 in F:
                a2 = -a6^3/F(3)-a5^2-2*a4*a6
                if isPP8(q,a6,a5,a4,a3,a2,a1): print(a6,a5,a4,a3,a2,a1)
\end{lstlisting}
\end{algorithm}

\begin{prop}
\label{prop:19}All PPs of degree $8$ in normalized form over $\mathbb{F}_{19}$
are exactly those of the form $x^{8}+\sum_{i=1}^{6}t^{8-i}a_{i}x^{i}$,
with $t\in\mathbb{F}_{19}^{*}$ and $(a_{6},a_{5},a_{4},a_{3},a_{2},a_{1})$
listed as follows: 
\begin{alignat*}{4}
 & (0,1,3,1,18,0), & \quad & (8,1,7,1,14,0), & \quad & (15,1,1,1,3,0), & \quad & (9,9,18,0,17,1),\\
 & (17,16,2,0,8,1), &  & (4,7,17,1,9,1), &  & (8,16,10,1,15,1), &  & (15,4,9,1,14,1),\\
 & (1,8,1,2,16,1), &  & (3,6,3,2,13,1), &  & (6,18,2,2,17,1), &  & (12,11,10,3,13,1),\\
 & (2,5,8,4,10,1), &  & (10,10,11,4,18,1), &  & (0,4,4,5,3,1), &  & (0,9,9,5,14,1),\\
 & (10,0,0,5,15,1), &  & (14,5,15,5,2,1), &  & (0,14,12,6,13,1), &  & (13,16,2,6,11,1),\\
 & (16,7,1,6,4,1), &  & (0,6,12,7,2,1), &  & (3,7,16,7,17,1), &  & (6,17,3,10,2,1),\\
 & (17,12,6,10,3,1), &  & (11,2,8,11,16,1), &  & (11,9,1,11,17,1), &  & (18,0,18,11,11,1),\\
 & (9,15,12,12,0,1), &  & (12,6,8,13,13,1), &  & (13,17,8,13,12,1), &  & (10,15,14,14,4,1),\\
 & (12,12,4,14,1,1), &  & (17,3,9,14,17,1), &  & (18,12,4,14,10,1), &  & (18,17,13,14,16,1),\\
 & (0,5,4,15,13,1), &  & (8,13,8,15,1,1), &  & (2,11,6,16,17,1), &  & (3,17,17,16,18,1),\\
 & (9,18,16,16,0,1), &  & (11,8,9,16,10,1), &  & (1,8,12,17,13,1), &  & (3,8,12,17,7,1),\\
 & (3,16,1,17,14,1), &  & (8,3,18,18,1,1), &  & (11,3,5,18,1,1), &  & (18,11,0,18,6,1).
\end{alignat*}

\end{prop}

\subsection{Case $q=13$}

Note that $a_{4}=-a_{6}^{2}/2$ since $\mathbf{HC}_{8}(13,2)=x_{6}^{2}+2x_{4}$.
We can make the following assumptions on $(a_{1},a_{2},\dots,a_{6})\in\mathbb{F}_{13}$,
by linear transformations if necessary.
\begin{itemize}
\item Assume $a_{1}\in\{0,1\}$ as $q-1=12$ is coprime to $7$.
\item When $a_{1}=0$, assume $a_{3}\in\{0,1\}$, as $q-1=12$ is coprime
to $5$.
\item When $a_{1}=a_{3}=0$, assume $a_{5}\in\{0,1,2,4\}$, as $\{1,2,4\}$
is a complete set of coset representatives of $\mathbb{F}_{13}^{*}/\{t^{3}:t\in\mathbb{F}_{13}^{*}\}$.
\end{itemize}
So Algorithm \ref{alg:13} lists PPs of degree $8$ in normalized
form over $\mathbb{F}_{13}$ up to linear transformations.

\begin{algorithm}[h]
\protect\caption{\label{alg:13} To list PPs of degree $8$ in normalized form over
$\mathbb{F}_{13}$}

\begin{lstlisting}
q = 13; F = GF(q)
for (a3,a1) in [(0,0),(1,0)]+[(a3,1) for a3 in F]:
    if (a3,a1)==(0,0): A5 = [0,1,2,4]
    else: A5 = F
    for a5 in A5:
        for a6 in F:
            a4 = -a6^2/F(2)
            for a2 in F:
                if isPP8(q,a6,a5,a4,a3,a2,a1): print(a6,a5,a4,a3,a2,a1)
\end{lstlisting}
\end{algorithm}

The output of Algorithm \ref{alg:13} prints 119 tuples $(a_{6},a_{5},\dots,a_{1})$,
among which three tuples $(4,2,5,0,4,0)$, $(10,2,2,0,4,0)$ and $(12,2,6,0,4,0)$
give three linearly related PPs of degree $8$. Indeed, for $t=3\in\mathbb{F}_{13}$,
we have $t^{3}=1$, $t^{-1}=9$ and  
\[
(10,2,2,0,4,0)=(4t^{2},2t^{3},5t^{4},0,4t^{6},0)=(12t^{-2},2t^{-3},6t^{-4},0,4t^{-6},0).
\]
No other linear transformation relations exist among the outputting
tuples. Therefore, there are exactly $117$ linearly related classes
of PPs of degree $8$ over $\mathbb{F}_{13}$, as listed in Proposition
\ref{prop:13} read off from the output of Algorithm \ref{alg:13}.
\begin{prop}
\label{prop:13}All PPs of degree $8$ in normalized form over $\mathbb{F}_{13}$
are exactly those of the form $x^{8}+\sum_{i=1}^{6}t^{8-i}a_{i}x^{i}$,
with $t\in\mathbb{F}_{13}^{*}$ and $(a_{6},a_{5},a_{4},a_{3},a_{2},a_{1})$
listed as follows: 
\begin{alignat*}{5}
 & (0,1,0,0,5,0), & \  & (0,1,0,0,7,0), & \  & (4,2,5,0,4,0), & \  & (0,4,0,0,9,0), & \  & (2,0,11,1,12,0),\\
 & (0,7,0,1,12,0), &  & (7,7,8,1,9,0), &  & (4,8,5,1,2,0), &  & (3,1,2,0,2,1), &  & (3,2,2,0,12,1),\\
 & (10,5,2,0,1,1), &  & (2,7,11,0,6,1), &  & (10,9,2,0,2,1), &  & (4,11,5,0,8,1), &  & (7,11,8,0,1,1),\\
 & (12,11,6,0,1,1), &  & (4,2,5,1,10,1), &  & (4,3,5,1,7,1), &  & (4,4,5,1,12,1), &  & (7,5,8,1,5,1),\\
 & (9,5,5,1,10,1), &  & (9,8,5,1,7,1), &  & (1,9,6,1,4,1), &  & (12,10,6,1,3,1), &  & (8,0,7,2,11,1),\\
 & (0,3,0,2,0,1), &  & (7,4,8,2,12,1), &  & (1,8,6,2,5,1), &  & (5,8,7,2,6,1), &  & (1,9,6,2,3,1),\\
 & (2,9,11,2,1,1), &  & (6,12,8,2,4,1), &  & (8,1,7,3,12,1), &  & (1,3,6,3,10,1), &  & (1,4,6,3,12,1),\\
 & (4,5,5,3,5,1), &  & (8,5,7,3,7,1), &  & (8,5,7,3,12,1), &  & (12,5,6,3,8,1), &  & (8,6,7,3,2,1),\\
 & (9,12,5,3,7,1), &  & (12,1,6,4,7,1), &  & (9,2,5,4,6,1), &  & (4,4,5,4,5,1), &  & (5,6,7,4,9,1),\\
 & (1,10,6,4,6,1), &  & (9,10,5,4,12,1), &  & (10,10,2,4,9,1), &  & (4,0,5,5,7,1), &  & (0,1,0,5,8,1),\\
 & (3,5,2,5,12,1), &  & (5,11,7,5,1,1), &  & (6,12,8,5,8,1), &  & (10,12,2,5,4,1), &  & (3,0,2,6,9,1),\\
 & (10,0,2,6,11,1), &  & (5,2,7,6,10,1), &  & (6,7,8,6,3,1), &  & (0,10,0,6,12,1), &  & (6,10,8,6,6,1),\\
 & (11,11,11,6,2,1), &  & (0,1,0,7,6,1), &  & (11,1,11,7,1,1), &  & (11,3,11,7,11,1), &  & (4,4,5,7,6,1),\\
 & (1,7,6,7,10,1), &  & (10,7,2,7,10,1), &  & (11,10,11,7,2,1), &  & (12,10,6,7,6,1), &  & (1,11,6,7,4,1),\\
 & (0,0,0,8,1,1), &  & (12,0,6,8,2,1), &  & (1,1,6,8,1,1), &  & (9,8,5,8,1,1), &  & (4,9,5,8,2,1),\\
 & (9,10,5,8,10,1), &  & (7,11,8,8,1,1), &  & (8,11,7,8,10,1), &  & (7,1,8,9,2,1), &  & (11,1,11,9,0,1),\\
 & (0,2,0,9,2,1), &  & (7,2,8,9,5,1), &  & (12,6,6,9,8,1), &  & (3,9,2,9,2,1), &  & (4,9,5,9,6,1),\\
 & (12,10,6,9,8,1), &  & (1,11,6,9,5,1), &  & (4,12,5,9,0,1), &  & (7,12,8,9,0,1), &  & (5,1,7,10,10,1),\\
 & (12,1,6,10,4,1), &  & (0,3,0,10,3,1), &  & (9,4,5,10,5,1), &  & (4,5,5,10,5,1), &  & (9,5,5,10,11,1),\\
 & (10,5,2,10,9,1), &  & (10,7,2,10,0,1), &  & (0,10,0,10,1,1), &  & (10,10,2,10,7,1), &  & (5,12,7,10,9,1),\\
 & (3,0,2,11,12,1), &  & (9,0,5,11,6,1), &  & (0,3,0,11,7,1), &  & (2,4,11,11,10,1), &  & (2,6,11,11,0,1),\\
 & (7,7,8,11,12,1), &  & (6,8,8,11,3,1), &  & (1,12,6,11,12,1), &  & (8,3,7,12,2,1), &  & (5,4,7,12,10,1),\\
 & (3,5,2,12,1,1), &  & (9,5,5,12,5,1), &  & (6,8,8,12,9,1), &  & (4,9,5,12,0,1), &  & (11,9,11,12,10,1),\\
 & (7,10,8,12,9,1), &  & (6,12,8,12,2,1),
\end{alignat*}

\end{prop}

\subsection{Case $q=11$}

Note that $a_{2}=-a_{5}^{2}/2-a_{4}a_{6}$ since $\mathbf{HC}_{8}(11,2)=x_{5}^{2}+2x_{4}x_{6}+2x_{2}$.
We can make the following assumptions on $(a_{1},a_{2},\dots,a_{6})\in\mathbb{F}_{11}$.
\begin{itemize}
\item Assume $a_{1}\in\{0,1\}$ as $q-1=10$ is coprime to $7$.
\item When $a_{1}=0$, assume $a_{5}\in\{0,1\}$, as $q-1=10$ is coprime
to $3$.
\item When $a_{1}=a_{5}=0$, assume $a_{6}\in\{0,1,2\}$, as $\{1,2\}$
is a complete set of coset representatives of $\mathbb{F}_{11}^{*}/\{t^{2}:t\in\mathbb{F}_{11}^{*}\}$.
\end{itemize}
So Algorithm \ref{alg:11} lists PPs of degree $8$ in normalized
form over $\mathbb{F}_{11}$ up to linear transformations.

\begin{algorithm}[h]
\protect\caption{\label{alg:11} To list PPs of degree $8$ in normalized form over
$\mathbb{F}_{11}$}

\begin{lstlisting}
q = 11; F = GF(q)
for (a5,a1) in [(0,0),(1,0)]+[(a5,1) for a5 in F]:
    if (a5,a1)==(0,0): A6 = [0,1,2]
    else: A6 = F
    for a6 in A6:
        for a4 in F:
            a2 = -a5^2/F(2)-a4*a6
            for a3 in F:
                if isPP8(q,a6,a5,a4,a3,a2,a1): print(a6,a5,a4,a3,a2,a1)  
\end{lstlisting}
\end{algorithm}

The output of Algorithm \ref{alg:11} prints $281$ tuples $(a_{6},a_{5},\dots,a_{1})$.
By Proposition \ref{prop:Eq8} and our assumptions, linear transformation
relations exist only among those with $a_{1}=a_{5}=0$, which are
indeed the first eight tuples in the output, corresponding to four
distinct linearly related classes. No other linear transformation
relations exist among the outputting tuples. Therefore, there are
exactly $277$ linearly related classes of PPs of degree $8$ over
$\mathbb{F}_{11}$, as listed in Proposition \ref{prop:11} read off
from the output of Algorithm \ref{alg:11}.
\begin{prop}
\label{prop:11}All PPs of degree $8$ in normalized form over $\mathbb{F}_{11}$
are exactly those of the form $x^{8}+\sum_{i=1}^{6}t^{8-i}a_{i}x^{i}$,
with $t\in\mathbb{F}_{11}^{*}$ and $(a_{6},a_{5},a_{4},a_{3},a_{2},a_{1})$
listed as follows: 
\begin{alignat*}{5}
 & (0,0,0,2,0,0), & \  & (0,0,0,4,0,0), & \  & (2,0,6,2,10,0), & \  & (2,0,7,3,8,0), & \  & (0,1,2,0,5,0),\\
 & (1,1,1,0,4,0), &  & (1,1,4,0,1,0), &  & (1,1,4,5,1,0), &  & (1,1,9,1,7,0), &  & (1,1,10,5,6,0),\\
 & (2,1,1,4,3,0), &  & (2,1,4,0,8,0), &  & (2,1,5,0,6,0), &  & (3,1,2,3,10,0), &  & (4,1,1,0,1,0),\\
 & (4,1,6,4,3,0), &  & (4,1,7,0,10,0), &  & (4,1,8,3,6,0), &  & (4,1,9,9,2,0), &  & (4,1,10,9,9,0),\\
 & (5,1,6,5,8,0), &  & (6,1,6,5,2,0), &  & (7,1,7,5,0,0), &  & (8,1,0,0,5,0), &  & (8,1,2,0,0,0),\\
 & (9,1,3,4,0,0), &  & (9,1,5,0,4,0), &  & (9,1,5,5,4,0), &  & (9,1,7,0,8,0), &  & (10,1,0,9,5,0),\\
 & (0,0,1,3,0,1), &  & (0,0,2,9,0,1), &  & (0,0,8,5,0,1), &  & (2,0,3,5,5,1), &  & (3,0,2,3,5,1),\\
 & (3,0,4,1,10,1), &  & (3,0,9,4,6,1), &  & (3,0,10,1,3,1), &  & (4,0,3,9,10,1), &  & (4,0,8,4,1,1),\\
 & (5,0,10,3,5,1), &  & (6,0,2,4,10,1), &  & (6,0,4,0,9,1), &  & (7,0,5,0,9,1), &  & (7,0,6,1,2,1),\\
 & (8,0,5,9,4,1), &  & (9,0,3,1,6,1), &  & (9,0,10,1,9,1), &  & (10,0,0,1,0,1), &  & (10,0,1,4,1,1),\\
 & (10,0,3,3,3,1), &  & (10,0,5,1,5,1), &  & (10,0,7,4,7,1), &  & (0,1,4,6,5,1), &  & (1,1,9,10,7,1),\\
 & (2,1,0,0,5,1), &  & (2,1,10,10,7,1), &  & (4,1,3,7,4,1), &  & (4,1,5,5,7,1), &  & (4,1,6,3,3,1),\\
 & (5,1,2,5,6,1), &  & (5,1,8,1,9,1), &  & (6,1,0,3,5,1), &  & (6,1,0,6,5,1), &  & (6,1,0,7,5,1),\\
 & (8,1,1,3,8,1), &  & (8,1,10,7,2,1), &  & (9,1,1,3,7,1), &  & (9,1,2,3,9,1), &  & (9,1,5,0,4,1),\\
 & (9,1,6,7,6,1), &  & (9,1,9,2,1,1), &  & (10,1,0,7,5,1), &  & (10,1,3,7,8,1), &  & (10,1,7,2,1,1),\\
 & (0,2,1,5,9,1), &  & (0,2,3,5,9,1), &  & (1,2,4,2,5,1), &  & (1,2,4,9,5,1), &  & (1,2,7,6,2,1),\\
 & (1,2,10,2,10,1), &  & (2,2,5,5,10,1), &  & (4,2,0,7,9,1), &  & (4,2,5,4,0,1), &  & (5,2,2,9,10,1),\\
 & (5,2,10,6,3,1), &  & (6,2,0,4,9,1), &  & (6,2,1,9,3,1), &  & (6,2,8,4,5,1), &  & (7,2,0,9,9,1),\\
 & (7,2,1,5,2,1), &  & (7,2,3,6,10,1), &  & (8,2,2,5,4,1), &  & (8,2,3,4,7,1), &  & (8,2,8,4,0,1),\\
 & (9,2,1,9,0,1), &  & (9,2,6,6,10,1), &  & (9,2,10,9,7,1), &  & (10,2,10,5,8,1), &  & (0,3,4,3,1,1),\\
 & (0,3,5,8,1,1), &  & (0,3,7,0,1,1), &  & (0,3,8,0,1,1), &  & (1,3,4,5,8,1), &  & (1,3,8,0,4,1),\\
 & (3,3,8,3,10,1), &  & (3,3,8,8,10,1), &  & (4,3,2,8,4,1), &  & (4,3,8,5,2,1), &  & (5,3,9,6,0,1),\\
 & (6,3,4,4,10,1), &  & (7,3,6,6,3,1), &  & (7,3,8,3,0,1), &  & (8,3,7,6,0,1), &  & (9,3,0,4,1,1),\\
 & (9,3,3,10,7,1), &  & (9,3,7,2,4,1), &  & (9,3,7,3,4,1), &  & (9,3,10,3,10,1), &  & (10,3,2,6,3,1),\\
 & (10,3,8,3,9,1), &  & (0,4,3,0,3,1), &  & (1,4,1,10,2,1), &  & (1,4,4,10,10,1), &  & (2,4,3,8,8,1),\\
 & (2,4,7,9,0,1), &  & (3,4,0,5,3,1), &  & (3,4,1,7,0,1), &  & (3,4,1,9,0,1), &  & (3,4,3,9,5,1),\\
 & (3,4,5,0,10,1), &  & (3,4,5,7,10,1), &  & (4,4,1,2,10,1), &  & (5,4,5,10,0,1), &  & (5,4,10,1,8,1),\\
 & (5,4,10,7,8,1), &  & (6,4,3,9,7,1), &  & (6,4,7,2,5,1), &  & (7,4,2,8,0,1), &  & (7,4,5,10,1,1),\\
 & (7,4,9,0,6,1), &  & (8,4,3,10,1,1), &  & (8,4,6,8,10,1), &  & (8,4,8,7,5,1), &  & (9,4,7,10,6,1),\\
 & (9,4,9,1,10,1), &  & (10,4,3,5,6,1), &  & (0,5,9,1,4,1), &  & (0,5,10,6,4,1), &  & (1,5,2,2,2,1),\\
 & (1,5,3,9,1,1), &  & (1,5,5,6,10,1), &  & (1,5,7,8,8,1), &  & (2,5,2,1,0,1), &  & (2,5,5,2,5,1),\\
 & (2,5,9,0,8,1), &  & (3,5,3,4,6,1), &  & (3,5,6,7,8,1), &  & (4,5,2,6,7,1), &  & (4,5,3,1,3,1),\\
 & (4,5,4,7,10,1), &  & (5,5,0,8,4,1), &  & (5,5,1,6,10,1), &  & (5,5,8,8,8,1), &  & (5,5,10,0,9,1),\\
 & (6,5,4,8,2,1), &  & (7,5,3,9,5,1), &  & (7,5,10,7,0,1), &  & (8,5,2,1,10,1), &  & (8,5,2,7,10,1),\\
 & (8,5,3,1,2,1), &  & (8,5,5,9,8,1), &  & (8,5,6,2,0,1), &  & (8,5,8,0,6,1), &  & (8,5,8,7,6,1),\\
 & (9,5,8,4,9,1), &  & (10,5,9,1,2,1), &  & (0,6,7,9,4,1), &  & (1,6,6,8,9,1), &  & (1,6,9,2,6,1),\\
 & (2,6,1,3,2,1), &  & (3,6,7,8,5,1), &  & (4,6,6,9,2,1), &  & (5,6,10,10,9,1), &  & (6,6,0,2,4,1),\\
 & (6,6,7,10,6,1), &  & (7,6,9,3,7,1), &  & (9,6,3,9,10,1), &  & (9,6,6,1,5,1), &  & (9,6,7,9,7,1),\\
 & (10,6,9,1,2,1), &  & (1,7,0,1,3,1), &  & (1,7,0,4,3,1), &  & (1,7,0,6,3,1), &  & (2,7,7,1,0,1),\\
 & (2,7,9,5,7,1), &  & (3,7,9,5,9,1), &  & (4,7,0,4,3,1), &  & (4,7,6,6,1,1), &  & (5,7,2,10,4,1),\\
 & (5,7,3,1,10,1), &  & (5,7,10,7,8,1), &  & (7,7,4,4,8,1), &  & (8,7,1,10,6,1), &  & (8,7,3,7,1,1),\\
 & (8,7,10,10,0,1), &  & (9,7,0,7,3,1), &  & (9,7,10,10,1,1), &  & (10,7,8,10,0,1), &  & (0,8,10,4,1,1),\\
 & (1,8,9,3,3,1), &  & (2,8,2,7,8,1), &  & (2,8,10,10,3,1), &  & (3,8,2,4,6,1), &  & (3,8,4,1,0,1),\\
 & (3,8,4,3,0,1), &  & (4,8,1,3,8,1), &  & (4,8,4,1,7,1), &  & (4,8,6,8,10,1), &  & (5,8,9,4,0,1),\\
 & (6,8,9,10,2,1), &  & (7,8,0,8,1,1), &  & (7,8,2,8,9,1), &  & (8,8,6,4,8,1), &  & (8,8,9,8,6,1),\\
 & (9,8,0,8,1,1), &  & (10,8,0,7,1,1), &  & (0,9,0,4,9,1), &  & (0,9,6,3,9,1), &  & (1,9,0,6,9,1),\\
 & (1,9,8,6,1,1), &  & (1,9,8,8,1,1), &  & (1,9,10,8,10,1), &  & (2,9,1,10,7,1), &  & (2,9,4,0,1,1),\\
 & (2,9,5,0,10,1), &  & (3,9,0,0,9,1), &  & (3,9,5,3,5,1), &  & (3,9,9,10,4,1), &  & (4,9,3,10,8,1),\\
 & (4,9,4,4,4,1), &  & (4,9,7,6,3,1), &  & (4,9,10,4,2,1), &  & (5,9,9,4,8,1), &  & (6,9,2,9,8,1),\\
 & (6,9,5,4,1,1), &  & (6,9,6,0,6,1), &  & (6,9,10,10,4,1), &  & (7,9,1,0,2,1), &  & (7,9,10,8,5,1),\\
 & (8,9,1,8,1,1), &  & (8,9,6,7,5,1), &  & (8,9,8,10,0,1), &  & (9,9,3,0,4,1), &  & (9,9,5,9,8,1),\\
 & (9,9,6,6,10,1), &  & (10,9,6,8,4,1), &  & (10,9,7,8,5,1), &  & (10,9,9,4,7,1), &  & (0,10,1,8,5,1),\\
 & (0,10,5,2,5,1), &  & (0,10,6,5,5,1), &  & (1,10,0,2,5,1), &  & (2,10,4,9,8,1), &  & (3,10,5,8,1,1),\\
 & (3,10,6,6,9,1), &  & (3,10,8,9,3,1), &  & (4,10,2,3,8,1), &  & (5,10,10,5,10,1), &  & (7,10,3,8,6,1),\\
 & (7,10,5,9,3,1), &  & (8,10,2,8,0,1), &  & (8,10,3,5,3,1), &  & (8,10,6,9,1,1), &  & (9,10,1,3,7,1),\\
 & (9,10,9,6,1,1), &  & (10,10,1,6,6,1).
\end{alignat*}

\end{prop}
\noindent \textbf{Acknowledgements.} This work was partially supported
by the Natural Science Foundation of Guangdong Province (No. 2018A030310080).
The author was also sponsored by the National Natural Science Foundation
of China (No. 11801579). Special thanks go to my lovely newborn daughter,
without whose birth should this paper have come out much earlier.


\begin{thebibliography}{19}
\expandafter\ifx\csname natexlab\endcsname\relax\def\natexlab#1{#1}\fi
\providecommand{\url}[1]{\texttt{#1}}
\providecommand{\href}[2]{#2}
\providecommand{\path}[1]{#1}
\providecommand{\DOIprefix}{doi:}
\providecommand{\ArXivprefix}{arXiv:}
\providecommand{\URLprefix}{URL: }
\providecommand{\Pubmedprefix}{pmid:}
\providecommand{\doi}[1]{\href{https://doi.org/#1}{\path{#1}}}
\providecommand{\Pubmed}[1]{\href{pmid:#1}{\path{#1}}}
\providecommand{\bibinfo}[2]{#2}
\ifx\xfnm\relax \def\xfnm[#1]{\unskip,\space#1}\fi
%Type = Article
\bibitem[{Chahal \& Ghorpade(2018)}]{ChahalGhorpade2018Carlitz}
\bibinfo{author}{Chahal, J.~S.}, \& \bibinfo{author}{Ghorpade, S.~R.}
  (\bibinfo{year}{2018}).
\newblock \bibinfo{title}{{Carlitz--Wan conjecture for permutation polynomials
  and Weill bound for curves over finite fields}}.
\newblock {\it \bibinfo{journal}{Finite Fields Appl.}\/},  {\it
  \bibinfo{volume}{54}\/}, \bibinfo{pages}{366--375}.
  \DOIprefix\doi{10.1016/j.ffa.2018.07.006}.
%Type = Article
\bibitem[{Cohen \& Fried(1995)}]{CohenFried1995Lenstra}
\bibinfo{author}{Cohen, S.~D.}, \& \bibinfo{author}{Fried, M.~D.}
  (\bibinfo{year}{1995}).
\newblock \bibinfo{title}{Lenstra's proof of the {C}arlitz-{W}an conjecture on
  exceptional polynomials: an elementary version}.
\newblock {\it \bibinfo{journal}{Finite Fields Appl.}\/},  {\it
  \bibinfo{volume}{1}\/}, \bibinfo{pages}{372--375}.
  \DOIprefix\doi{10.1006/ffta.1995.1027}.
%Type = Misc
\bibitem[{Decker et~al.(2019)Decker, Greuel, Pfister \&
  Sch\"onemann}]{Singular}
\bibinfo{author}{Decker, W.}, \bibinfo{author}{Greuel, G.-M.},
  \bibinfo{author}{Pfister, G.}, \& \bibinfo{author}{Sch\"onemann, H.}
  (\bibinfo{year}{2019}).
\newblock \bibinfo{title}{{\sc Singular} {4-1-2} --- {A} computer algebra
  system for polynomial computations}.
\newblock \bibinfo{howpublished}{\url{http://www.singular.uni-kl.de}}.
%Type = Article
\bibitem[{Dickson(1896/97)}]{Dickson1897analytic}
\bibinfo{author}{Dickson, L.~E.} (\bibinfo{year}{1896/97}).
\newblock \bibinfo{title}{The analytic representation of substitutions on a
  power of a prime number of letters with a discussion of the linear group}.
\newblock {\it \bibinfo{journal}{Ann. of Math.}\/},  {\it
  \bibinfo{volume}{11}\/}, \bibinfo{pages}{65--120}.
  \DOIprefix\doi{10.2307/1967217}.
%Type = Article
\bibitem[{Ding \& Yuan(2006)}]{DingYuan2006family}
\bibinfo{author}{Ding, C.}, \& \bibinfo{author}{Yuan, J.}
  (\bibinfo{year}{2006}).
\newblock \bibinfo{title}{A family of skew {H}adamard difference sets}.
\newblock {\it \bibinfo{journal}{J. Combin. Theory Ser. A}\/},  {\it
  \bibinfo{volume}{113}\/}, \bibinfo{pages}{1526--1535}.
  \DOIprefix\doi{10.1016/j.jcta.2005.10.006}.
%Type = Misc
\bibitem[{Fan(2018{\natexlab{a}})}]{Fan2019PP7}
\bibinfo{author}{Fan, X.} (\bibinfo{year}{2018}{\natexlab{a}}).
\newblock \bibinfo{title}{{A classification of permutation polynomials of
  degree $7$ over finite fields}}.
\newblock \href{http://arxiv.org/abs/1812.02080}{\tt arXiv:1812.02080}.
%Type = Misc
\bibitem[{Fan(2018{\natexlab{b}})}]{Fan2019Weil}
\bibinfo{author}{Fan, X.} (\bibinfo{year}{2018}{\natexlab{b}}).
\newblock \bibinfo{title}{{The Weil bound and non-exceptional permutation
  polynomials over finite fields}}.
\newblock \href{http://arxiv.org/abs/1811.12631}{\tt arXiv:1811.12631}.
%Type = Misc
\bibitem[{Fan(2019)}]{Fan2019PP8p2}
\bibinfo{author}{Fan, X.} (\bibinfo{year}{2019}).
\newblock \bibinfo{title}{{Permutation polynomials of degree $8$ over finite
  fields of characteristic $2$}}.
\newblock \href{http://arxiv.org/abs/1903.10309}{\tt arXiv:1903.10309}.
%Type = Article
\bibitem[{Fried et~al.(1993)Fried, Guralnick \&
  Saxl}]{FriedGuralnickSaxl1993Schur}
\bibinfo{author}{Fried, M.~D.}, \bibinfo{author}{Guralnick, R.}, \&
  \bibinfo{author}{Saxl, J.} (\bibinfo{year}{1993}).
\newblock \bibinfo{title}{Schur covers and {C}arlitz's conjecture}.
\newblock {\it \bibinfo{journal}{Israel J. Math.}\/},  {\it
  \bibinfo{volume}{82}\/}, \bibinfo{pages}{157--225}.
  \DOIprefix\doi{10.1007/BF02808112}.
%Type = Article
\bibitem[{von~zur Gathen(1991)}]{Gathen1991values}
\bibinfo{author}{von~zur Gathen, J.} (\bibinfo{year}{1991}).
\newblock \bibinfo{title}{Values of polynomials over finite fields}.
\newblock {\it \bibinfo{journal}{Bull. Austral. Math. Soc.}\/},  {\it
  \bibinfo{volume}{43}\/}, \bibinfo{pages}{141--146}.
  \DOIprefix\doi{10.1017/S0004972700028860}.
%Type = Article
\bibitem[{Hayes(1967)}]{Hayes1967geometric}
\bibinfo{author}{Hayes, D.~R.} (\bibinfo{year}{1967}).
\newblock \bibinfo{title}{A geometric approach to permutation polynomials over
  a finite field}.
\newblock {\it \bibinfo{journal}{Duke Math. J.}\/},  {\it
  \bibinfo{volume}{34}\/}, \bibinfo{pages}{293--305}.
  \DOIprefix\doi{10.1215/S0012-7094-67-03433-3}.
%Type = Article
\bibitem[{Hermite(1863)}]{Hermite1863sur}
\bibinfo{author}{Hermite, C.} (\bibinfo{year}{1863}).
\newblock \bibinfo{title}{Sur les fonctions de sept lettres}.
\newblock {\it \bibinfo{journal}{C. R. Acad. Sci. Paris}\/},  {\it
  \bibinfo{volume}{57}\/}, \bibinfo{pages}{750--757}.
%Type = Article
\bibitem[{Hou(2015)}]{Hou2015survey}
\bibinfo{author}{Hou, X.-d.} (\bibinfo{year}{2015}).
\newblock \bibinfo{title}{Permutation polynomials over finite fields---a survey
  of recent advances}.
\newblock {\it \bibinfo{journal}{Finite Fields Appl.}\/},  {\it
  \bibinfo{volume}{32}\/}, \bibinfo{pages}{82--119}.
  \DOIprefix\doi{10.1016/j.ffa.2014.10.001}.
%Type = Article
\bibitem[{Kemper(2002)}]{Kemper2002calculation}
\bibinfo{author}{Kemper, G.} (\bibinfo{year}{2002}).
\newblock \bibinfo{title}{The calculation of radical ideals in positive
  characteristic}.
\newblock {\it \bibinfo{journal}{J. Symbolic Comput.}\/},  {\it
  \bibinfo{volume}{34}\/}, \bibinfo{pages}{229--238}.
  \DOIprefix\doi{10.1006/jsco.2002.0560}.
%Type = Article
\bibitem[{Li et~al.(2010)Li, Chandler \&
  Xiang}]{LiChandlerXiang2010permutation}
\bibinfo{author}{Li, J.}, \bibinfo{author}{Chandler, D.~B.}, \&
  \bibinfo{author}{Xiang, Q.} (\bibinfo{year}{2010}).
\newblock \bibinfo{title}{Permutation polynomials of degree 6 or 7 over finite
  fields of characteristic 2}.
\newblock {\it \bibinfo{journal}{Finite Fields Appl.}\/},  {\it
  \bibinfo{volume}{16}\/}, \bibinfo{pages}{406--419}.
  \DOIprefix\doi{10.1016/j.ffa.2010.07.001}.
%Type = Book
\bibitem[{Lidl \& Niederreiter(1997)}]{LidlNiederreiter1997book}
\bibinfo{author}{Lidl, R.}, \& \bibinfo{author}{Niederreiter, H.}
  (\bibinfo{year}{1997}).
\newblock {\it \bibinfo{title}{Finite fields}\/} volume~\bibinfo{volume}{20} of
  {\it \bibinfo{series}{Encyclopedia of Mathematics and its Applications}\/}.
\newblock (\bibinfo{edition}{2nd} ed.).
\newblock \bibinfo{publisher}{Cambridge University Press, Cambridge}.
\newblock \DOIprefix\doi{10.1017/CBO9780511525926} \bibinfo{note}{with a
  foreword by P. M. Cohn}.
%Type = Book
\bibitem[{Mullen(2013)}]{Mullen2013Handbook}
\bibinfo{editor}{Mullen, G.~L.} (Ed.) (\bibinfo{year}{2013}).
\newblock {\it \bibinfo{title}{Handbook of finite fields}\/}.
\newblock Discrete Mathematics and its Applications (Boca Raton).
\newblock \bibinfo{publisher}{CRC Press, Boca Raton, FL}.
\newblock \DOIprefix\doi{10.1201/b15006}.
%Type = Manual
\bibitem[{{The Sage Developers}(2017)}]{SageMath}
\bibinfo{author}{{The Sage Developers}} (\bibinfo{year}{2017}).
\newblock {\it \bibinfo{title}{{SageMath, the Sage Mathematics Software
  System}}\/}.
\newblock \URLprefix \url{http://www.sagemath.org}.
  \DOIprefix\doi{10.5281/zenodo.593563}.
%Type = Incollection
\bibitem[{Wan(1993)}]{Wan1993padic}
\bibinfo{author}{Wan, D.~Q.} (\bibinfo{year}{1993}).
\newblock \bibinfo{title}{A {$p$}-adic lifting lemma and its applications to
  permutation polynomials}.
\newblock In {\it \bibinfo{booktitle}{Finite fields, coding theory, and
  advances in communications and computing ({L}as {V}egas, {NV}, 1991)}\/} (pp.
  \bibinfo{pages}{209--216}).
\newblock \bibinfo{publisher}{Dekker, New York} volume \bibinfo{volume}{141} of
  {\it \bibinfo{series}{Lecture Notes in Pure and Appl. Math.}\/}.
\newblock \URLprefix \url{http://www.math.uci.edu/~dwan/lift.pdf}.

\end{thebibliography}
\end{document}